\documentclass[11pt]{article}
\usepackage{amsmath, amssymb, amscd, amsthm, amsfonts}
\usepackage{graphicx}
\usepackage{hyperref}
\usepackage{setspace}
\usepackage{subfig}
\usepackage{comment}
\usepackage{float}
\usepackage{bbm}
\usepackage{algorithm}
\usepackage{algorithmic}
\usepackage{bigints}
\usepackage{mathtools}
\usepackage[nottoc,notlot,notlof]{tocbibind}
\usepackage[margin=1.2in]{geometry}

\title{Edgeworth expansion for Bernoulli weighted mean}
\author{Pierre-Louis Cauvin\thanks{Criteo AI Lab.} \thanks{Grenoble INP - UGA, Ensimag, 38400 Saint-Martin-d'Hères, France.}\\ \href{mailto:pierre-louis.cauvin@grenoble-inp.org}{pierre-louis.cauvin@grenoble-inp.org}}
\date{}

\newtheorem{theorem}{Theorem}[section]
\newtheorem{defn}[theorem]{Definition}
\newtheorem{lemma}[theorem]{Lemma}

\newtheorem{proposition}[theorem]{Proposition}
\newtheorem{remark}[theorem]{Remark}
\newtheorem{example}[theorem]{Example}

\newcommand{\rr}{\mathbb{R}}

\newcommand{\E}[1]{\mathbb{E} \left[ \mspace{1mu} #1 \mspace{1mu} \right]}
\renewcommand{\P}[1]{\mathbb{P} \left( #1 \right)}
\renewcommand{\O}{\mathcal{O}}

\newcommand{\ceil}[1]{\left\lceil #1 \right\rceil}

\usepackage{listings}
\usepackage{color}

\definecolor{dkgreen}{rgb}{0,0.6,0}
\definecolor{gray}{rgb}{0.5,0.5,0.5}
\definecolor{mauve}{rgb}{0.58,0,0.82}

\begin{document}

\maketitle

\begin{abstract}
In this work, we derive an Edgeworth expansion for the Bernoulli weighted mean $\hat{\mu} = \frac{\sum_{i=1}^n Y_i T_i}{\sum_{i=1}^n T_i}$ in the case where $Y_1, \dots, Y_n$ are i.i.d. non semi-lattice random variables  and $T_1, \dots, T_n$ are Bernoulli distributed random variables with parameter $p$. We also define the notion of a semi-lattice distribution, which gives a more geometrical equivalence to the classical Cramér's condition in dimensions bigger than 1. Our result provides a first step into the generalization of classical Edgeworth expansion theorems for random vectors that contain both semi-lattice and non semi-lattice variables, in order to prove consistency of bootstrap methods in more realistic setups, for instance in the use case of online AB testing.
\end{abstract}

\section{Introduction and Main Result}

In probability and statistic applied problems, one often have to deal with the limit of sample distribution  when its size gets bigger, in order to construct asymptotically consistent estimators or confidence intervals for instance. The most widely used theorem in this case is the classical Central Limit theorem (CLT) and its variation the Delta method, which guarantee that a lot of statistics converge in law to a Gaussian distribution. However, the CLT only gives a convergence result but not its speed of convergence nor the error terms of the approximation when we only deal with low finite sample size. The first point can be tackme with Berry-Esseen bounds, stating that under some regularity and finite moments conditions the convergence is of order $\O(1/\sqrt{n})$ where $n$ is the sample size. Whereas for the second point we need a precise analysis of error terms of the approximation, what motivates the theory of Edgeworth expansions \cite{hall1997bootstrap}.

We say that a statistics $T(Y_1, \dots, Y_n)$ of some sample $(Y_1, \dots, Y_n)$ admits an Edgeworth expansion of order $q$ if there exist polynomials $p_1, \dots, p_q$ whose coefficients depend only on the moments of $(Y_1, \dots, Y_n)$ such that
\begin{align*}
    \P{T \leq x} = \Phi(x) + \sum_{j=1}^{q} n^{-j/2}p_j(x)\phi(x) + \O(n^{-(q+1)/2})
\end{align*}
uniformly in $x$, where $\Phi$ and $\phi$ are the cumulative distribution function and the probability distribution function of a standard Gaussian distribution respectively.

The theory of Edgeworth expansions was vastly explored since the first formal asymptotic expansions were proposed in a paper by Tchebycheff in 1890 \cite{Tchebycheff}, and later on at the end of the $20^{th}$ century when contemporary statisticians found its usefulness to study properties of empirical statistical methods. One pioneer of this revival of interest was Peter Hall, who single-handedly proved consistency, orders of convergence and coverage error of confidence intervals for the B. Efron's bootstrap \cite{EfronBoot}, one of the most used statistical approximation method, based on general existence properties of Edgeworth expansions \cite{hall1997bootstrap}.

However, most of these general existence theorems hold only for function of empirical means of i.i.d. samples and under what is called the Cramér's condition, which denies the use of semi-lattice (see Definition \ref{semi-lattice}) random variables, including many of the classical discrete distributions. Some progress have been made since the start of the years 2000s in order to generalize and weaken such conditions, for instance \cite{EdgGeneralStats} shows the existence of Edgeworth expansions for a broader class of statistics which contains $U$- and $L$-statistics, but still need the Cramér's condition, whereas some other authors have proposed some Edgeworth expansions for purely lattice triangular array \cite{EdgLatticeTriangular}, bounded integer valued random variables \cite{EdgIndepBoundedInt}, or have given a weaker Cramér's condition to allow discrete but non semi-lattice distribution \cite{WeakCramer}.

Therefore, in this work we propose a first step into the generalization of Edgeworth expansions for some semi-lattice random variables, which are of vastly practical interests. More precisely, we prove the following theorem which provides an Edgeworth expansion for the Bernoulli weighted mean of non semi-lattice random variables, which is purely semi-lattice due to the mixture of Bernoulli and non semi-lattice random variables as we will see in Section \ref{Preliminary}.

\begin{theorem}[Edgeworth expansion for Bernoulli weighted mean]\label{MainTheorem}
    Suppose that $Y_1, \dots, Y_n$ are i.i.d. non semi-lattice random variables with mean $\mu$ and variance $\sigma^2$. Let $T_1, \dots, T_n$ be i.i.d. Bernoulli distributed random variables with parameter $p > 0$, and write $N = \sum_{i=1}^n T_i$. Then, if $Y$ has at least $q$ finite moments, we have the Edgeworth expansion
    \begin{align}
        \P{\sqrt{N}\frac{\hat{\mu} - \mu}{\sigma} \leq x} = \Phi(x) + \sum_{j=1}^{q-2} n^{-j/2} p^*_j(x) \phi(x) + \O\left(n^{-(q-1)/2}\right)
    \end{align}
    uniformly in $x$, where $\hat{\mu} = \begin{cases}
        \frac{\sum_{i=1}^n Y_i T_i}{\sum_{i=1}^n T_i} & \text{if $\sum_{i=1}^n T_i \neq 0$}\\
        \hfil
        0 & \text{otherwise}
    \end{cases}$ is the Bernoulli weighted mean, and $p^*_j$'s are polynomials of degree at most $3j$ depending only on the first $j+2$ moments of $Y_1$ and on moments of $T$.
\end{theorem}

Bernoulli weighted means are the basis of many statistical online estimators, such as uplift modeling in AB testing with various applications in health or online advertising, where we want to estimate the different behaviors between a treatment group and a control group, where each individual is assigned at random based on a Bernoulli random variable. As such, we obtain estimators which are difference or ratios of Bernoulli weighted means, motivating our work. Another application of Bernoulli weighted means appears when we have a big sample, e.g. the whole population, and we want to estimate statistical properties of smaller samples taken at random with a Bernoulli variable flipped for each individual. 

This paper is organized as follow : in Section \ref{Preliminary}, we first state preliminary definitions and results that will be useful later on. More precisely, we give the classical Edgeworth expansion theorem \cite{hall1997bootstrap}, and we provide an equivalence for the multivariate Cramér's condition as the definition of a semi-lattice random variable. We further use this equivalence to show that the classical theorem doesn't hold for Bernoulli weighted means. Then, in Section \ref{proof} we provide a proof for Theorem \ref{MainTheorem}, based on intermediate results using a conditioning with respect to the Bernoulli random variables $T_1, \dots, T_n$ and asymptotic expansions for inverse binomial moments. We also give an easier proof for the bound of the expectation of arbitrary powers of a binomial random variable than the ones that can be found in \cite{Inverse_Moments_2000} and \cite{Znidaric_2009}.

\section{Preliminary results}\label{Preliminary}

We will extensively use the notion of uniformly bounded equality, which we recall its definition below, and we prove that it is stable with respect to addition.

\begin{defn}[Uniformly bounded]
    For two sequences of functions $(f_n)_{n\in\mathbb{N}}, (g_n)_{n\in\mathbb{N}} : \rr \to \rr$ and a sequence $(u_n)_{n\in\mathbb{N}}$ of $n$, we say that $f_n(x) = g_n(x) + \O(u_n)$ uniformly in $x$ if there exists an integer $N > 0$ such that
    \begin{align*}
        \underset{x \in \rr}{\sup} \, | f_n(x) - g_n(x) | \leq C u_n
    \end{align*}
    for all $n \geq N$, where $C > 0$ is a constant which does not depend on $x$.
\end{defn}

\begin{proposition}[Linearity of uniform bounds]\label{LinBounds}
    Let $(f^1_n)_{n\in\mathbb{N}}, (f^2_n)_{n\in\mathbb{N}}, (g^1_n)_{n\in\mathbb{N}}, (g^2_n)_{n\in\mathbb{N}} : \rr \to \rr$ be sequences of functions, and let $(u^1_n)_{n\in\mathbb{N}}, (u^2_n)_{n\in\mathbb{N}}$ be sequences of $n$. Suppose that $f^1_n(x) = g^1_n(x) + \O(u^1_n)$ and $f^2_n(x) = g^2_n(x) + O(u^2_n)$ both uniformly in $x$. Then
    \begin{align*}
        f^1_n(x) + f^2_n(x) = g^1_n(x) + g^2_n(x) + \O(u^1_n + u^2_n)
    \end{align*}
    uniformly in $x$.
\end{proposition}
\begin{proof}
    By triangular inequality, we have
    \begin{align*}
        \underset{x \in \rr}{\sup} \, | f^1_n(x) + f^2_n(x) - g^1_n(x) - g^2_n(x) | &\leq \underset{x \in \rr}{\sup} \, | f^1_n(x) - g^1_n(x) | + \underset{x \in \rr}{\sup} \, | f^2_n(x) - g^2_n(x) |\\
        &\leq C^1 u^1_n + C^2 u^2_n\\
        &\leq C (u^1_n + u^2_n)
    \end{align*}
    where $C = \max(C^1, C^2)$.
\end{proof}

The following classical theorem on Edgeworth expansion for functions of sample mean is due to \cite{hall1997bootstrap} and can also be found in \cite{shao1995jackknife}. 

\begin{theorem}\label{Edgeworth_Ttm}
    Let $(X_n)$ be a sequence of i.i.d. random $d$-vectors and let $g$ be a Borel measurable function on $\mathbb{R}^d$ taking real values. Suppose that the following conditions hold
    \begin{enumerate}
        \item $X_1$ has at least $q \geq 3$ finite moments,
        \item $g$ is $\mathcal{C}^{p-1}$ in a neighborhood of $\mu = \E{X_1}$,
        \item $\nabla g(\mu) \neq 0$ where $\nabla$ denotes the gradient,
        \item (Cramér's condition) $\underset{||\mathbf{t}|| \to \infty}{\limsup} ~  |\psi_{X_1}(\mathbf{t})| < 1$ where $\psi_{X_{1}}(\mathbf{t}) = \E{\exp(i \mathbf{t} \cdot X_{1})}$ is the characteristic function of $X_1$.
    \end{enumerate}
    Then, we have the Edgeworth expansion 
    \begin{align*}
        \P{\sqrt{n} \, \dfrac{g\left(\bar{X}_n\right) - g(\mu)}{\sigma} \leq x} = \Phi(x) + \sum_{j=1}^{q-2} n^{-j/2} p_j(x) \phi(x) + \mathcal{O} \left( n^{-(q-1)/2} \right)
    \end{align*}
    uniformly in $x$, where $\sigma^2 = \nabla g(\mu)^T Var(X) \nabla g(\mu)$ and $p_j$ is a polynomial of degree at most $3j$ whose coefficients depend only on the first $j+2$ moments of $X_1$ and the partial derivatives of $g$ at $\mu$.
\end{theorem}

This theorem is a powerful tool to prove existence of Edgeworth expansions for a variety of different statistics. However, one of its hypothesis if quite restrictive in many applied problems : Cramér's condition.  Cramér's condition is strongly linked with the notion of lattice distributions, that we define below.

\begin{defn}[Lattice distribution]
    A random $d$-vector $X$ is said to have a lattice distribution if it takes values almost surely in a lattice $\{ \mathbf{x_0} + k \boldsymbol{\delta}, ~ k \in \mathbb{Z} \}$ spanned by $\mathbf{x_0}, \boldsymbol{\delta} \in \mathbb{R}^d$.
\end{defn}

\begin{example}~
    \begin{enumerate}
        \item The Bernoulli distribution in $d=1$ is a lattice distribution, because it takes values in $\{0, 1\}$ which is included in the lattice $\{0 + k \times 1, k \in \mathbb{Z} \}$.
        \item If we take the uniform distribution on the Boolean hypercube $\{0,1\}^d$, then it is a lattice distribution of dimension $d$.
        \item All absolutely continuous distributions with respect to Lebesgue measure are not lattice.
    \end{enumerate}
\end{example}

For $d=1$, the Cramér's condition is equivalent to the fact that $X$ has a non lattice distribution. For instance, Cramér's condition is always true when $X$ is absolutely continuous with respect to Lebesque measure. However, in the case where $d > 1$ we no longer have the equivalence between Cramér's condition and classical lattice distributions. 

This is what motivates the following definition of semi-lattice distribution, which is equal to lattice distribution for $d=1$, and that we will show is also equivalent to Cramér's condition in any finite dimension. This definition also allows for a more geometrical view of multivariate Cramér's condition : there is no direction such that the distribution is lattice.

\begin{defn}[Semi-lattice distribution]\label{semi-lattice}
    A random $d$-vector $X$ is said to have a semi-lattice distribution if there exists a vector $\mathbf{t^*} \in \mathbb{R}^d \setminus \{0\}$ such that $\mathbf{t^*} \cdot X$ takes values almost surely in a lattice $\{x_0 + k\delta, ~k\in\mathbb{Z}\}$ spanned by $x_0, \delta \in \mathbb{R}$.
\end{defn}

\begin{remark}~
    \begin{enumerate}
        \item When $d=1$, all semi-lattice distributions are discrete, but not all discrete distributions are semi-lattice (e.g. $X \in \{e, 3, \pi\}$ uniformly).
        \item When $d > 1$, semi-lattice distributions are not necessarily discrete. For example, the random vector $(N, 4N)$ where $N \sim \mathcal{N}(0, 1)$ is semi-lattice as $(-4, 1) \cdot (N, 4N) = 0$.
    \end{enumerate}
\end{remark}

Now we prove the equivalence between semi-lattice behavior and Cramér's condition. More precisely, Proposition \ref{Prop4.3} shows that all semi-lattice distributions do not verify Cramér's condition, and on the other hand a consequence of Proposition \ref{Prop2.7} is that if the Cramér's condition doesn't hold, then the distribution is necessary semi-lattice.

\begin{proposition}\label{Prop4.3}
    If a random $d$-vector $X$ has a semi-lattice distribution, then there exists a vector $\mathbf{t^*} \in \mathbb{R}^d \setminus \{0\}$ such that $r \mapsto |\psi_X(r\mathbf{t^*})|$ is periodic. In particular, it means that we have $\underset{||\mathbf{t}|| \to \infty}{\limsup} ~  |\psi_{X}(\mathbf{t})| = 1$.
\end{proposition}
\begin{proof}
    Suppose that we have a $\mathbf{t^*} \neq 0$ such that $\mathbf{t^*} \cdot X \in \{x_0 + k\delta, ~k\in\mathbb{Z}\}$ almost surely. Then, denoting $\mu_X$ the law of $X$, we can write the function $r \mapsto |\psi_X(r\mathbf{t^*})|$ as 
    \begin{align*}
        |\psi_X(r\mathbf{t^*})| = \left| \int_{\mathbb{R}^d} e^{i r\mathbf{t^*}\cdot \mathbf{x}} d\mu_X(\mathbf{x}) \right| &= \left| \sum_{k\in\mathbb{Z}} e^{ir(x_0 + k\delta)} \P{\mathbf{t^*} \cdot X = x_0 + k\delta} \right|\\
        &= \left| \sum_{k\in\mathbb{Z}} e^{irk\delta} \P{\mathbf{t^*} \cdot X = x_0 + k\delta} \right|
    \end{align*}
    Let us fix $r^* \in \mathbb{R}$, and consider the particular case $r_n = \frac{2n\pi}{\delta} + r^*$. Using the equality above, we find that
    \begin{align*}
        |\psi_X(r_n\mathbf{t^*})| = \left| \sum_{k\in\mathbb{Z}} e^{2ink}e^{ir^*k\delta} \P{\mathbf{t^*} \cdot X = x_0 + k\delta} \right| &= \left| \sum_{k\in\mathbb{Z}} e^{ir^*k\delta} \P{\mathbf{t^*} \cdot X = x_0 + k\delta} \right|\\
        &= |\psi_X(r^*\mathbf{t^*})|
    \end{align*}
    Therefore the function $r \mapsto |\psi_X(r\mathbf{t^*})|$ is $\frac{2\pi}{\delta}$-periodic. In particular, setting $\mathbf{t_n} = \frac{2n\pi}{\delta} \mathbf{t^*}$, it means that $\underset{||\mathbf{t}|| \to \infty}{\limsup} ~  |\psi_{X}(\mathbf{t})| \geq \underset{n \to \infty}{\lim} ~ |\psi_{X}(\mathbf{t_n})| = 1$ and so $\underset{||\mathbf{t}|| \to \infty}{\limsup} ~  |\psi_{X}(\mathbf{t})| = 1$.
\end{proof}

\begin{proposition}\label{Prop2.7}
    Let $X$ be a random $d$-vector admitting a characteristic function $\phi_X$. Suppose that there exists a $t^* \neq 0$ such that $|\phi_X(t^*)| = 1$. Then the distribution of $X$ is semi-lattice.
\end{proposition}
\begin{proof}
    If $|\phi_X(t^*)|=1$, it means that there exists a $\theta \in \mathbb{R}$ such that $\phi_X(t^*)=e^{i\theta}$. Rewriting $\phi_X$ by its expression and dividing both sides by $e^{i\theta}$, it implies that $\E{e^{i(t^* \cdot X - \theta)}} = 1$. Therefore, if we use the decomposition of $e^{it}$ as cosinus and sinus, we obtain easily that $\E{\cos(t^* \cdot X - \theta)}=1$. But $\cos(t^* \cdot x - \theta) - 1 \geq 0$ for all vector $x \in \mathbb{R}^d$, thus this equality is possible if and only if $\cos(t^* \cdot X - \theta) = 1$ almost surely. And so we have necessarily $t^* \cdot X = \theta + k2\pi$ for $k \in \mathbb{Z}$, meaning that $X$ is semi-lattice.
\end{proof}

This equivalence between semi-lattice and Cramér's condition gives a convenient way to verify if we can apply Theorem \ref{Edgeworth_Ttm} on some random vector without having to manipulate its characteristic function. For instance, suppose that we would like to use Theorem \ref{Edgeworth_Ttm} to show existence of an Edgeworth expansion for the Bernoulli weighted mean $\hat{\mu} = \frac{\sum_{i=1}^n Y_i T_i}{\sum_{i=1}^n T_i}$. This statistic can indeed be expressed as a function of samples means by noticing that $\hat{\mu} = \overline{YT} / \overline{T} = g(\overline{YT}, \overline{T})$ where $g : x,y \mapsto x/y$, so we are in the framework of the theorem. However, according to the following Lemma \ref{Lemma8.1}, the distribution of $(YT, T)$ doesn't follow Cramér's condition, hence we unfortunately cannot apply Theorem \ref{Edgeworth_Ttm} for Bernoulli weighted means.

\begin{lemma}\label{Lemma8.1}
    Let $Y$ be an arbitrary random variable and $T$ a Bernoulli distributed random variable with parameter $p$. Then $\underset{||\mathbf{t}|| \to \infty}{\limsup} ~  |\psi_{(YT,T)}(\mathbf{t})| = 1$ where $\psi_{(YT,T)}$ is the characteristic function of the vector $(YT, T)$.
\end{lemma}
\begin{proof}
    Setting $t^* = (0, 1)$, we obtain $t^* \cdot (YT, T) = T$ which takes values almost surely in the lattice $\{0,1\}$, meaning that the vector $(YT, T)$ has a semi-lattice distribution. As such, we can use Prop \ref{Prop4.3} to conclude that $\underset{||\mathbf{t}|| \to \infty}{\limsup} ~  |\psi_{(YT,T)}(\mathbf{t})| = 1$.
\end{proof}

\section{Proof of Theorem \ref{MainTheorem}}\label{proof}

We recall that we define the Bernoulli weighted average $\hat{\mu}$ as
\begin{align*}
    \hat{\mu} = \begin{cases}
        \dfrac{\sum_{i=1}^n Y_i T_i}{\sum_{i=1}^n T_i} & \text{if $\sum_{i=1}^n T_i \neq 0$}\\
        \hfil
        0 & \text{otherwise}
    \end{cases}
\end{align*}
where $Y_1, \dots, Y_n$ are i.i.d non semi-lattice random variables, and $T_1, \dots, T_n$ are Bernoulli random variables with parameter $p > 0$. Furthermore, we let $N = \sum T_i$ and $Z = \sqrt{N}\frac{\hat{\mu} - \mu}{\sigma}$ where $\mu$ and $\sigma^2$ are the mean and variance of $Y$ respectively. For any $k = 1, \dots, n$, we also define $\overline{Y}_k = \frac{1}{k} \sum_{i=1}^k$ the empirical mean of sample size $k$.

Before proceeding through the proof of Theorem \ref{MainTheorem}, we first need to state and prove some intermediate useful results. The following Lemma allows us to reduce the problem of expanding the whole c.d.f. of $Z$, which depend both on the $Y_i$'s and the $T_i$'s, to a problem of approximation for each $k = 1, \dots, n$ the easier c.d.f. $F_{Z_k}(x) := \P{ \sqrt{k} \frac{\overline{Y}_k - \mu}{\sigma} \leq x}$ involving only the $Y_i$'s.

\begin{lemma}\label{Lemma8.2}
    For every real $\alpha$, the cumulative distribution function $F_Z$ of $Z$ can be written as
        \begin{align*}
            F_Z(x) = \sum_{k=1}^n \binom{n}{k} p^k (1-p)^{n-k} \P{\sqrt{k} \dfrac{\overline{Y}_k - \mu}{\sigma} \leq x} + \O(n^\alpha)
        \end{align*}
    uniformly in $x$.
\end{lemma}
\begin{proof}
    Let $q_1, \dots, q_n \in \{0,1\}$ where not all $q_i$'s are $0$. Then, by conditioning we have
    \begin{align}\label{Lemma8.2.1}
        \P{Z \leq x ~|~ T_1=q_1, \dots, T_n=q_n} = \P{\dfrac{\sqrt{\sum q_i}}{\sigma} \left[\frac{1}{\sum q_i} \sum_{q_i = 1} Y_i - \mu\right] \leq x}
    \end{align}
    Now, the idea is to notice that $\sum_{i=1}^{\sum q_i} Y_i$ has the same distribution as $\sum_{q_i=1} Y_i$ due to the fact that the $Y_i$'s are independent and of the same law. Therefore, if we define $k = \sum q_i$, we can rewrite \eqref{Lemma8.2.1} as 
    \begin{align}\label{condBern}
        \P{Z \leq x ~|~ T_1=q_1, \dots, T_n=q_n} = \P{ \sqrt{k} \dfrac{\overline{Y}_k - \mu}{\sigma} \leq x}
    \end{align}
    Now, we observe that the right-hand side of \eqref{condBern} only depends on $k$, meaning that the conditional c.d.f. is invariant by permutations of the $q_i$'s. Thus, by defining $N = \sum T_i$, we obtain
    \begin{align*}
        \P{Z \leq x ~|~ N=k} = \P{ \sqrt{k} \dfrac{\overline{Y}_k - \mu}{\sigma} \leq x}
    \end{align*}
    In the case where $q_1=\dots=q_n=0$, we have instead $\P{Z \leq x ~|~ N=0} = \P{0 \leq x} = \mathbb{I}\{x \geq 0\}$.
    Hence, by the law of total expectation, the c.d.f. of $Z$ can be expressed as 
    \begin{align*}
        F_Z(x) &= \sum_{k=0}^n \P{N=k} \P{Z \leq x ~|~ N = k}\\
        &= \sum_{k=1}^n \binom{n}{k} p^k (1-p)^{n-k} \P{ \sqrt{k} \dfrac{\overline{Y}_k - \mu}{\sigma} \leq x} + \mathbb{I}\{x \geq 0\} (1-p)^n
    \end{align*}
    where we have used that $N$ follows a binomial distribution with parameters ($n$, $p$). To obtain the result it only remains to show that for all real $\alpha$, $\mathbb{I}\{x \geq 0\} (1-p)^n = \O(n^\alpha)$ uniformly in $x$, which is immediate because $|\mathbb{I}\{x \geq 0\}| \leq 1$ and $(1-p)^n$ decays to $0$ faster than $n^\alpha$.
\end{proof}

Hence, now we can focus ourselves on the existence of Edgeworth expansions for each c.d.f. $F_{Z_k}$, which is what we prove in the following Lemma.

\begin{lemma}\label{EdgMean}
    Suppose that $Y$ has at least $q$ finite moments. Then, for all positive integer $k$, $\overline{Y}_k$ admits the Edgeworth expansion
    \begin{align*}
        \P{ \sqrt{k} \dfrac{\overline{Y}_k - \mu}{\sigma} \leq x} = \Phi(x) + \sum_{j=1}^{q-2} k^{-j/2} p_j(x) \phi(x) + \O\left(k^{-(q-1)/2}\right)
    \end{align*}
    uniformly in $x$, where $p_j$ is a polynomial of degree at most $3j$ and depending only on the first $j-2$ moments of $Y$.
\end{lemma}
\begin{proof}
    This a special case of Theorem \ref{Edgeworth_Ttm} where $g = x \mapsto x$. Another direct proof of this result is conducted in \cite[Chapter~2.2]{hall1997bootstrap}.
\end{proof}

Putting together this Edgeworth expansion result with the expression of Lemma \ref{Lemma8.2}, we will find ourselves with sums of the form $\sum_{k=1}^n k^\alpha \binom{n}{k} p^k (1-p)^k$ with $\alpha$ real, that we call the Bernoulli sum of $k^\alpha$. In what follows, we will provide bounds for these Bernoulli sums in order to control exactly the error terms of the asymptotic expansions.

First, we prove a simple bound on the expectation of powers of binomial random variables, result that was already shown in \cite{Inverse_Moments_2000} and \cite{Znidaric_2009}, but we provide an easier and shorter proof based on the Kullback-Leibler upper-bound for binomial random variables.

\begin{proposition}\label{Prop:BinomBound}
    Let $N$ be a binomial random variable with parameters $(n, p)$ where $p \in (0, 1/2]$. Then, for all $\alpha \in \rr$ we have
    \begin{align*}
        \E{(N+1)^\alpha} = \O(n^\alpha)
    \end{align*}
\end{proposition}
\begin{proof}~\\
    \underline{If $\alpha \geq 0$} We have $N \leq n$, so $\E{(N+1)^\alpha} \leq (n+1)^\alpha$ and we obtain immediately the result by noticing that $(n+1)^\alpha = \O(n^\alpha)$.\\
    \underline{If $\alpha = -\beta < 0$} In order to prove the result, we will use the following classical inequality obtained by optimizing the Chernoff bound for a binomial distribution : for all $\delta \in (0, p)$, we have
    \begin{align}\label{BinKL}
        \P{N \leq \delta n} \leq e^{-n\mathcal{D}(\delta||p)}
    \end{align}
    where $\mathcal{D}$ is the Kullback-Leibler divergence between two Bernoulli distributions of parameters $\delta$ and $p$ respectively. Turning back to our former expectation, by conditioning on the event $\{N \leq \delta n\}$ we obtain
    \begin{align*}
        \E{(N+1)^\alpha} = \P{N \leq \delta n} \E{\dfrac{1}{(N+1)^\beta} ~|~ N \leq \delta n} + \P{N > \delta n}  \E{\dfrac{1}{(N+1)^\beta} ~|~ N > \delta n}
    \end{align*}
    Now, we can use \eqref{BinKL} to bound $\P{N \leq \delta n}$, and we also have $\E{\frac{1}{(N+1)^\beta} ~|~ N \leq \delta n} \leq 1$ due to $N + 1 \geq 1$, $\P{N > \delta n} \leq 1$ because it's a probability, and $\E{\frac{1}{(N+1)^\beta} ~|~ N > \delta n} \leq \frac{1}{(\delta n + 1)^\beta}$. All in all, we finaly obtain
    \begin{align*}
        \E{(N+1)^\alpha} \leq e^{-n\mathcal{D}(\delta||p)} + (\delta n + 1)^\alpha
    \end{align*}
    where the right-hand side is a $\O(n^\alpha)$ because the Kullback-Leibler divergence is strictly positive for $\delta < p$, proving the proposition.
\end{proof}

An immediate consequence of this bound is the following Lemma, controlling the speed of convergence of the Bernoulli sum of $k^\alpha$.

\begin{lemma}\label{BoundO}
    Let $p \in (0, 1/2]$ and $n \geq 1$. Then, for all $\alpha \in \rr$
    \begin{align*}
        \sum_{k=1}^n k^\alpha \binom{n}{k} p^k (1-p)^{n-k} \, \overset{(1)}{=} \, np \E{(N^* + 1)^{\alpha - 1}} \, \overset{(2)}{=} \, \O(n^\alpha)
    \end{align*}
    where $N^*$ is a binomial random variable with parameters $(n-1, p)$.
\end{lemma}
\begin{proof}
    Expanding the expectation and using a change of index, we obtain
    \begin{align*}
        \E{(N^* + 1)^{\alpha - 1}} &= \sum_{k=0}^{n-1} (k+1)^{\alpha - 1} \binom{n-1}{k} p^k (1-p)^{n-1-k}\\
        &= \sum_{k=1}^{n} k^{\alpha - 1} \binom{n-1}{k-1} p^{k-1} (1-p)^{n-k}\\
        &= \sum_{k=1}^{n} k^{\alpha - 1} \dfrac{k}{n} \binom{n}{k} \dfrac{1}{p} p^k (1-p)^{n-k}\\
        &= \dfrac{1}{np} \sum_{k=1}^n k^\alpha \binom{n}{k} p^k (1-p)^{n-k}
    \end{align*}
    which proves equality $(1)$. Now, by Proposition \ref{Prop:BinomBound} it is straightforward that $\E{(N^* + 1)^{\alpha - 1}} = \O\left((n-1)^{\alpha-1}\right) = \O(n^{\alpha-1})$, and so equality $(2)$ holds immediately by multiplying both side by $np$.
\end{proof}

However, in order to prove the Edgeworth expansion of Theorem \ref{MainTheorem} we will also need an even stronger result on binomial random variables, namely an asymptotic expansion for the inverse binomial moments.

\begin{defn}[Inverse binomial moment]
    For any real $\alpha > 0$, we define the inverse binomial moment $f_\alpha(n)$ of order $\alpha$ as 
    \begin{align*}
        f_\alpha(n) = \sum_{k=1}^n \dfrac{1}{k^\alpha} \binom{n}{k} p^k (1-p)^{n-k}
    \end{align*}
\end{defn}

The following result was first stated and proved in \cite{Znidaric_2009} with a general closed form for the constants of the expansion. 

\begin{proposition}\label{BinomExpansion}
    The inverse binomial moment $f_\alpha(n)$ can be expanded in terms of inverse powers of $(np)$, meaning that for any integer $q$ 
    \begin{align*}
        f_\alpha(n) = \dfrac{1}{(np)^\alpha} \left( \sum_{k=0}^{q-1} C_{\alpha,k} \dfrac{1}{(np)^k} \right) + \O\left(\dfrac{1}{(np)^{\alpha + q}}\right)
    \end{align*}
    where $C_{\alpha,k}$ are constants depending only on $\alpha$ and $p$.
\end{proposition}
\begin{proof}
    See \cite{Znidaric_2009} for the original proof using expansion of the generating function of central moments, or \cite{Cichon_Golebiewski_2012} for another generalized approach based on Bernstein polynomials.
\end{proof}

Finally, we need this last Lemma in order to construct asymptotic expansions for sums of inverse binomial moments.

\begin{lemma}\label{BoundSum}
    For every $j = 1, \dots, q-2$, let $p_j$ be a polynomial of degree at most $3j$ and depending only on the first $j+2$ moments of some r.v. $Y$. Furthermore, let $\phi$ be the p.d.f. of the standard Gaussian distribution. Then, we have
    \begin{align*}
        \sum_{j=1}^{q-2} f_{j/2}(n) p_j(x) \phi(x) = \sum_{j=1}^{q-2} n^{-j/2} p^*_j(x) \phi(x) + \O\left( n^{-(q-1)/2} \right)
    \end{align*}
    uniformly in $x$, where $p^*_j$'s are polynomials of degree at most $3j$ and depending only on the first $j+2$ moments of $Y$ and on moments of $T$.
\end{lemma}
\begin{proof}
    For each $j = 1, \dots, q-2$, we can use Proposition \ref{BinomExpansion} to obtain
    \begin{align*}
        \left| f_{j/2}(n) - \sum_{i=0}^{q-2 + \ceil{j/2}} C_{j,i} (np)^{-j/2 + i} \right| \leq K_j n^{-(q-1)/2}
    \end{align*}
    where $K_j > 0$ is a constant. This comes from the fact that $-j/2 + q-2 + \ceil{j/2} < q-1$ and $-j/2 + q-1 + \ceil{j/2} \geq q-1$. Therefore, for all $x \in \rr$, we have also that 
    \begin{align*}
        \left| f_{j/2}(n) p_j(x) \phi(x) - \sum_{i=0}^{q-2 + \ceil{j/2}} C_{j,i} (np)^{-j/2 + i} p_j(x) \phi(x) \right| \leq \left| p_j(x) \phi(x) \right| K_j n^{-(q-1)/2}.
    \end{align*}
    However $p_j(x)\phi(x)$ is bounded uniformly because $\phi(x)$ has an exponential decay, and so we can put the bound inside the constant $K_j$. Hence, for each $j = 1, \dots, q-2$ we conclude that
    \begin{align}\label{UniformBoundForEachj}
        f_{j/2}(n) p_j(x) \phi(x) = \sum_{i=0}^{q-2 + \ceil{j/2}} C_{j,i} (np)^{-j/2 + i} p_j(x) \phi(x) + \O\left( n^{-(q-1)/2} \right).
    \end{align}
    uniformly in $x$. Now we can sum equality \eqref{UniformBoundForEachj} over every $j$, giving 
    \begin{align*}
        \sum_{j=1}^{q-2} f_{j/2}(n) p_j(x) \phi(x) = \sum_{j=1}^{q-2} ~ \sum_{i=0}^{q-2 + \ceil{j/2}} C_{j,i} (np)^{-j/2 + i} p_j(x) \phi(x) + \O\left( n^{-(q-1)/2} \right)
    \end{align*}
    still uniformly in $x$ due to Proposition \ref{LinBounds}. Noticing that for every couple $(i,j)$ we have $-j/2 + i = -k/2$ where $1 \leq k \leq q-2$ and summing over each of such $k$, we can therefore rewrite the equality as
    \begin{align*}
        \sum_{j=1}^{q-2} f_{j/2}(n) p_j(x) \phi(x) = \sum_{j=1}^{q_2} n^{-j/2} p^*_j(x) \phi(x) + \O\left( n^{-(q-1)/2} \right)
    \end{align*}
    unformly in $x$, where $p^*_j$ is a polynomial given by 
    \begin{align*}
        p^*_j(x) = 
        \begin{cases}
            p^{-j/2} \sum_{i=0}^{\frac{j-1}{2}} C_{2i+1, \frac{j-1}{2} -i} ~ p_{2i+1}(x) & \text{if $j$ odd}\\
            p^{-j/2} \sum_{i=1}^{\frac{j}{2}} C_{2, \frac{j}{2} -i} ~ p_{2i}(x) & \text{if $j$ even}
        \end{cases}
    \end{align*}
    With this expression, it's easy to see that $p^*_j$ has the same degree as $p_j$, that it only depends on the first $j+2$ moments of $Y$ because it is a linear combination of all polynomials $p_k$ where $k \leq j$ is of same parity as $j$, and that it also depends on the moments of $T$ because of the constants $C_{j,i}$ depending on $p$. Thus proving the result.
\end{proof}

After all these intermediate results, we are finally ready to prove Theorem \ref{MainTheorem}.

\begin{proof}[Proof of Theorem \ref{MainTheorem}]
    Using the Edgeworth expansions of Lemma \ref{EdgMean} in each term of the decomposition of the c.d.f. $F_Z$ provided by Lemma \ref{Lemma8.2}, for $\alpha = -(q-1)/2$ we obtain
    \begin{align}\label{CDF_Edg}
        F_Z(x) &= \left[1 - (1-p)^n \right] \Phi(x) + \sum_{j=1}^{q-2} \left( \sum_{k=1}^n k^{-\frac{j}{2}} \binom{n}{k} p^k (1-p)^{n-k} \right) p_j(x) \phi(x)\nonumber \\
        &+ \O\left( \sum_{k=1}^n k^{-\frac{q-1}{2}} \binom{n}{k} p^k (1-p)^{n-k} \right) + \O(n^{-(q-1)/2})
    \end{align}
    uniformly in $x$ due to Proposition \ref{LinBounds}. Here the $\left[1 - (1-p)^n \right] \Phi(x)$ term comes from the fact that $\sum_{k=1}^n \binom{n}{k} p^k (1-p)^{n-k} = 1 - (1-p)^n$ since the sum starts at $1$ and not at $0$. Now, the idea is to bound every term involving binomial sums in order to find an Edgeworth expansion as powers of $n^{-1/2}$ instead of powers of $k^{-1/2}$.
    \\
    First, we have that $|(1-p)^n \Phi(x)| \leq (1-p)^n$ which is uniformly bounded in $x$ by $n^{-(q-1)/2}$, and so we can put it inside the already present $\O(n^{-(q-1)/2})$.
    \\
    Next, according to Lemma \ref{BoundO} we have $\sum_{k=1}^n k^{-\frac{q-1}{2}} \binom{n}{k} p^k (1-p)^{n-k} = \O(n^{-(q-1)/2})$, hence allowing us to rewrite \eqref{CDF_Edg} as
    \begin{align}\label{CDF_Edg_2}
        F_Z(x) &= \Phi(x) + \sum_{j=1}^{q-2} \left( \sum_{k=1}^n k^{-\frac{j}{2}} \binom{n}{k} p^k (1-p)^{n-k} \right) p_j(x) \phi(x) + \O(n^{-(q-1)/2})
    \end{align}
    uniformly in $x$. Finally, it only remains to apply Lemma \ref{BoundSum} in order to bound uniformly the last sum term of \eqref{CDF_Edg_2}, giving 
    \begin{align*}
        F_Z(x) &= \Phi(x) + \sum_{j=1}^{q-2} n^{-j/2} p^*_j(x) \phi(x) + \O(n^{-(q-1)/2})
    \end{align*}
    where $p^*_j$ are polynomials verifying the required conditions, thus proving our main theorem.
\end{proof}

\section*{Acknowledgments}

This work was done as part of a 3-month internship in the Causality team of Criteo AI Lab. The author would like to thank all members of this team for their hospitality, especially Matthieu Martin for having supervised this internship and reviewed parts of this paper several times.

\newpage 

\bibliographystyle{utphys}
\bibliography{bibliography}

\providecommand{\href}[2]{#2}\begingroup\raggedright\begin{thebibliography}{10}

\bibitem{WeakCramer}
J.~Angst and G.~Poly, ``{A weak Cramér condition and application to Edgeworth
  expansions},'' {\em Electronic Journal of Probability} {\bfseries 22} no.~59,
  (2017) 1 -- 24.

\bibitem{EdgLatticeTriangular}
A.~Bock, ``Edgeworth expansions for lattice triangular arrays.'' 2014.

\bibitem{Cichon_Golebiewski_2012}
J.~Cichoń and Z.~Golebiewski, ``On bernoulli sums and bernstein polynomials,''
  {\em Discrete Mathematics \& Theoretical Computer Science} {\bfseries DMTCS
  Proceedings vol. AQ, 23rd Intern. Meeting on Probabilistic, Combinatorial,
  and Asymptotic Methods for the Analysis of Algorithms (AofA'12)} (Jan., 2012)
  179--190.

\bibitem{Inverse_Moments_2000}
F.~Cribari-Neto, N.~L. Garcia, and K.~L.~P. Vasconcellos, ``A note on inverse
  moments of binomial variates,'' {\em Brazilian Review of Econometrics}
  {\bfseries 20} no.~2, (2000) .

\bibitem{EdgIndepBoundedInt}
D.~Dolgopyat and Y.~Hafouta, ``Edgeworth expansions for independent bounded
  integer valued random variables,'' 2020.

\bibitem{EfronBoot}
B.~Efron, ``Bootstrap methods: Another look at the jackknife,'' {\em The Annals
  of Statistics} {\bfseries 7} no.~1, (1979) 1--26.

\bibitem{hall1997bootstrap}
P.~Hall, {\em The Bootstrap and Edgeworth Expansion}.
\newblock Springer Series in Statistics. Springer New York, 1997.

\bibitem{EdgGeneralStats}
B.-Y. Jing and Q.~Wang, ``A unified approach to edgeworth expansions for a
  general class of statistics,'' {\em Statistica Sinica} {\bfseries 20} no.~2,
  (2010) 613--636.

\bibitem{shao1995jackknife}
J.~Shao and D.~Tu, {\em The Jackknife and Bootstrap}.
\newblock Springer Series in Statistics. Springer New York, 1995.

\bibitem{Tchebycheff}
P.~Tchebycheff, ``{Sur deux théorèmes relatifs aux probabilités},'' {\em
  Acta Mathematica} {\bfseries 14} no.~none, (1890) 305 -- 315.

\bibitem{Znidaric_2009}
M.~Znidaric, ``Asymptotic expansion for inverse moments of binomial and poisson
  distributions,'' {\em The Open Statistics {\&} Probability Journal}
  {\bfseries 1} no.~1, (Jan, 2009) 7--10.

\end{thebibliography}\endgroup
\nocite{*}

\end{document}